\documentclass{amsart}
\usepackage{amsthm, amsmath, amssymb}

\theoremstyle{definition}

\newtheorem{con}{Conjecture}
\newtheorem{thm}{Theorem}[section]
\newtheorem{lemma}[thm]{Lemma}
\newtheorem{cor}[thm]{Corollary}

\newtheorem{example}{Example}[section]

\theoremstyle{remark}
\newtheorem{remark}[thm]{Remark}

\usepackage[colorlinks]{hyperref}
\usepackage{extarrows}
\usepackage[all]{xy}
\usepackage{geometry}
\newgeometry{left=3.18cm, right=3.18cm, top=3.54cm, bottom=4.54cm}
\usepackage{indentfirst}
\usepackage{graphics}
\graphicspath{{Figures/}}
\usepackage{color}
\usepackage{graphicx}
\usepackage{xcolor}

\usepackage{bm}
\usepackage{enumerate}

\title{Chirally Cosmetic Surgeries on Kinoshita-Terasaka and Conway knot families}
\author{Xiliu Yang}
\address{LMAM, School of Mathematics Sciences, Peking University, Beijing, P.R. China}
\email{1801110020@pku.edu.cn}
\date{\today}

\begin{document}
\maketitle

\begin{abstract}
In this note, we prove that a nontrivial Kinoshita-Terasaka or Conway knot does not admit chirally cosmetic surgeries, by calculating the finite type invariant of order 3.
\end{abstract}

\section{Introduction}

Let $ K $ be a knot in $ S^3 $ and $ r $ be a number in $ \mathbb{Q} \cup \{ \infty \} $, we denote by $ S^3_r(K) $ the manifold obtained by the Dehn surgery along $ K $ with slope $ r $. Two surgeries along $ K $ with distinct slopes $ r $ and $ s $ are called \emph{purely cosmetic} if $ S^3_r(K) \cong S^3_s(K) $, and called \emph{chirally cosmetic} if $ S^3_r(K) \cong - S^3_s(K) $. Here $ M \cong N $ means that $ M $ and $ N $ are homeomorphic as oriented manifolds, and $ -M $ represents the manifold $ M $ with opposite orientation.

The (purely) Cosmetic Surgery Conjecture \cite[Problem 1.81(A)]{Kir95}, \cite[Conjecture 6.1]{Gor90} asserts that if a knot $ K $ is nontrivial, then it does not admit purely cosmetic surgeries. This conjecture has been studied in many cases using different obstructions. For instance, if $ K $ admits purely cosmetic surgeries, then the surgery slope cannot be $ \infty $ \cite{GL89}, the normalized Alexander polynomial $ \Delta_K(t) $ of $ K $ satisfies $ \Delta_K''(1) = 0 $ \cite{BL90}, the Jones polynomial $ V_K(t) $ satisfies $ V''_K(1) = V'''_K (1) = 0 $ \cite{IW19}, and the finite type invariants satisfies $ v_2(K) = v_3(K) = 0 $ \cite{IW19}. Some other constraints are given, for example, by using the LMO invariants \cite{Ito17}, the quantum $ SO(3) $-invariant \cite{Det21}. Besides the above criteria, Heegaard Floer homology, as well as its immersed curve version, has been particularly effective for this conjecture. Combining the work of Ozsv\'ath and Szab\'o \cite{OS11}, Ni and Wu \cite{NW15}, and Hanselman \cite{Han18}, we know that if two distinct slopes $ r $ and $ s $ are purely cosmetic, then $ r = -s $, and the set $ \{r,s\} $ can only be $ \{ \pm 2 \} $ or $ \{ \pm 1/p \} $ for some integer $ p $. This conjecture has been verified for many knots, including, Seifert genus one knots \cite{Wang06}, cable knots \cite{Tao18}, composite knots \cite{Tao19},  2-bridge knots \cite{IJMS19}, 3-braid knots \cite{Var20}, pretzel knots \cite{SS20}, and knots with at most 17 crossings \cite{Han18, Det21}. Indeed, the purely cosmetic surgeries are really rare. Specifically, given $ b > 0 $, there are only finitely many knots with braid index $ b $ that possibly admit purely cosmetic surgeries \cite{Ito21}.

On the other hand, the chirally case is rather complicated since there are two known families of chirally cosmetic surgeries for knots in $ S^3 $:
\begin{enumerate}
	\item[(A).] For an amphicheiral knot $ K $ and a slope $ r $, we have $ S^3_r(K) \cong -S^3_{-r}(K) $.
	\item[(B).] For a $ (2, k) $-torus knot $ K $, we have $ S^3_r(K) \cong -S^3_s(K) $, where $ \{r, s \} = \{\frac{2k^2(2m + 1)}{k(2m+1) + 1}, \frac{2k^2(2m + 1)}{k(2m+1) - 1}\} $ for some integer $ m $ \cite{Mat92}.
\end{enumerate}
With the exception of the above two cases, no knot was found to have chirally cosmetic surgeries. The conjecture for chirally case states as follows:
\begin{con}{\cite[Conjecture 1]{IIS21}}\label{CCSC1}
	Suppose $ K $ is not amphichiral and is not a $ (2, k) $-torus knot, then $ K $ does not admit chirally cosmetic surgeries.
\end{con}

The conjecture has been verified for alternating genus one knots \cite{IIS19}, alternating odd pretzel knots \cite{Var20,Var21}, a certain family of the positive Whitehead doubles \cite{Var21}, and cable knots with some additional assumptions \cite{Ito21c}. In this short note, we prove the following:
\begin{thm}\label{main_thm}
	Any nontrivial Kinoshita-Terasaka and Conway knot does not admit chirally cosmetic surgeries.
\end{thm}

The purely cosmetic surgeries on knots of both two families are ruled out in \cite{BGLX20}, and we provide an alternative proof. Our proof is based on the calculation of the constraint $ O(K)$ defined in \cite{IIS21}. Let $ a_{2i}(K) $ be the coefficient of the $ z^{2i} $-term of the Conway polynomial $ \nabla_K(z) $, and $ v_3(K) $ be the finite type invariant of order 3, $ O(K) $ is defined as:
$$ O(K) \triangleq \left\{ \begin{aligned}
& \left|\frac{7a_2(K)^2 - a_2(K) - 10a_4(K)}{4v_3(K)}\right|, & v_3(K) \neq 0;\\
& \infty, &\text{otherwise}.
\end{aligned} \right. $$

We appeals the following obstruction theorem:

\begin{thm}{\cite[Theorem 1.10]{IIS21}}\label{key}
	A knot $ K $ has no chirally cosmetic surgeries if $ O(K) \leq 2 $.
\end{thm}

\subsection*{Acknowledgements.} The author would like to thank Professor Jiajun Wang for helpful discussions. The author also thanks Zhujun Cao and Cheng Chang for corrections.

\section{Prove of the main result}
In \cite{KT57}, Kinoshita and Terasaka constructed a family of knots $ KT_{r,n} $ parametrized by integers $ r $ and $ n $. These knots are obtained from a diagram of the four-stranded pretzel links $ P(r+1, -r, r, -r-1) $ by introducing $ 2n $ twists, as shows in Figure \ref{KTC}. There are some redundancies in these knots. Specifically, $ KT_{r,n} $ is isotopic to the unknot if and only if $ r \in \{0, \pm 1, -2 \} $ or $ n = 0 $. By turning the knot inside out, one can observe a symmetry which identifies $ KT_{r,n} $ and $ KT_{-r-1, n} $. Finally, we note that the mirror image of $ KT_{r,n} $ is $ KT_{r,-n} $.

\begin{figure}[htb]
	\centering
	\includegraphics[height=5cm]{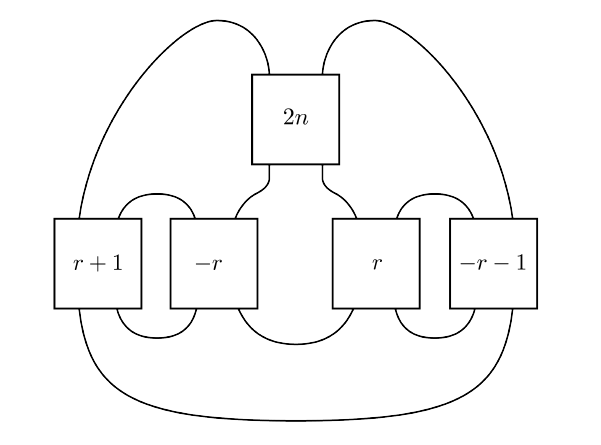}
	\quad
	\includegraphics[height=5cm]{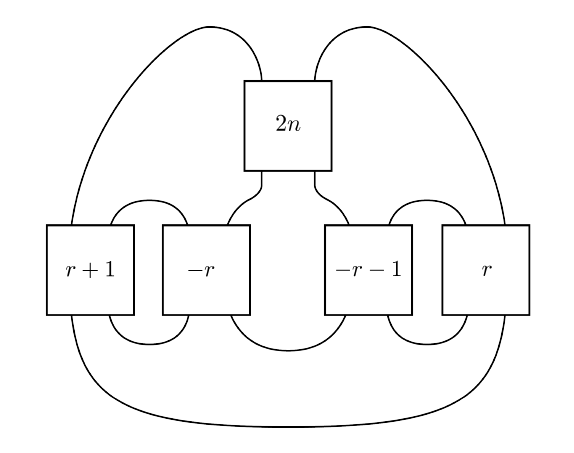}
	\caption{The diagram of Kinoshita-Terasaka knots $ KT_{r,n} $ (left) and Conway knots $ C_{r,n} $ (right). The number $ r $ in the boxes means $ r $-times half twists.}
	\label{KTC}
\end{figure}

The Conway knot $ C_{r, n} $ shows in Figure \ref{KTC}, is obtained from $ KT_{r,n} $ by a mutation. These knots also have a similar construction as $ KT_{r,n} $, only using the four-stranded pretzel links $ P(r+1,-r,-r-1, r) $  instead of $ P(r+1, -r, r, -r-1) $. Therefore, Conway and Kinoshita-Terasaka knots satisfy many of the same relations. In particular, $ C_{r, n} $ is isotopic to the unknot iff $ r \in \{0, \pm 1, -2 \} $ or $ n = 0 $, $ C_{r,n} = C_{-r-1, n} $, and $ C^*_{r,n} = C_{r, -n} $. \\

What makes these two families special is that they have trivial Alexander polynomials $ \Delta_K(t) $, as well as Conway polynomials $ \nabla_K(z) $, since $ \Delta_K(t) = \nabla_K(t^{1/2} - t^{-1/2}) $. We note that their Conway polynomials can be computed directly by skein relation at $ 2n $-twist part from those of pretzel link computed in \cite[Theorem 3.2]{KL07}. By definition, $ a_2(K) = a_4(K) = 0 $. It remains to find $ v_3(K) $.
\begin{lemma}\label{v_3}
	$ v_3(KT_{r,n}) = v_3(C_{r,n}) = -\frac{nk(k+1)}{4} $, here $ k = \lfloor\frac{r}{2}\rfloor $.
\end{lemma}
\begin{proof}
	By the symmetry identifying, there is no loss of generality in assuming $ r \geq 0 $ and $ n \geq 0 $. We compute for Conway knot $ C_{r,n} $ first.
	
	Note that
	$$ v_3(K) = -\frac{1}{144}V_K'''(1) - \frac{1}{48} V_K''(1) =  -\frac{1}{24}j_3(K), $$
	where $ V_K(t) $ is the Jones polynomial, and $ j_n(K) $ is the coefficient of $ h^n $ in $ V_K(e^h) $ of $ K $, by putting $ t = e^h $. Here we use another knot invariant $ w_3(K) $ defined by Lescop in \cite{Les09}. The advantage is $ w_3 $ satisfies a crossing change formula
	\begin{equation}\label{w_3}
	w_3(K_+) - w_3(K_-) = \frac{a_2(K') + a_2(K'')}{2} - \frac{a_2(K_+) + a_2(K_-) + {\rm lk}^2(K', K'')}{4},
	\end{equation} 
	where $ (K_+, K_-, K' \cup K'') $ is a skein triple consisting of two knots $ K_{\pm} $ and a two-component link $  K' \cup K''$, cf. \cite[Proposition 7.2]{Les09} and \cite{IW19}. On the other hand, a formula from Hoste \cite[Theorem 1]{Hos85} states that
	\begin{equation}\label{a_2}
	{\rm lk}(K', K'') = a_2(K_+) - a_2(K_-).
	\end{equation}
	In our case, by smoothing at $ 2n $-twist part, $ K_+ $ is $ C_{r,n} $ and $ K_- $ is $ C_{r,n-1} $, both of which have trivial Conway polynomial; and $ K', K'' $ are two components of pretzel link $ P(r+1, -r, -r-1, r) $. If $ r $ is even, the two components $ K' $ and $ K'' $ are torus knots $ T_{2, r+1} $ and $ T_{2, -r-1} $, respectively; when $ r $ is odd, they are $ T_{2, r} $ and $ T_{2, -r} $. We compute $ a_2(T_{2,r}) $ when $ T_{2,r} $ is a knot, i.e., $ r = 2k+1 $ is odd. By equation (\ref{a_2}) again,
	$$ a_2(T_{2,2k+1}) - a_2(T_{2,2k-1}) = {\rm lk}(K_1, K_2), $$
	here $ K_1 $ and $ K_2 $ are two components of the torus link $ T_{2,2k} $, thus with linking number $ k $. Notice that $ T_{2,1} $ is the unknot, so that $ a_2(T_{2,1}) = 0 $. Therefore, it is easy to see
	$$ a_2(T_{2,2k+1}) = \frac{k(k+1)}{2}. $$
	Since $ \nabla_K(z) = \nabla_{K^*}(z) $ holds for knot $ K $ and its mirror $ K^* $, we can obtain
	$$ w_3(C_{r,n}) - w_3(C_{r,n-1}) = \frac{a_2(K') + a_2(K'')}{2} = a_2(K') = \left\{ \begin{aligned}
	& \frac{k(k+1)}{2}, & r = 2k; \\
	& \frac{k(k+1)}{2}, & r = 2k+1.
	\end{aligned}
	\right. $$
	Note that $ w_3(K) = \frac{1}{72}V'''_K(1) + \frac{1}{24}V''_K(1) = -2v_3(K) $, cf. \cite[Lemma 2.2]{IW19}. When $ n = 0 $, $ C_{r,0} $ is the unknot, so that $ w_3(C_{r,n-1}) = -2 v_3(\text{unknot}) = 0 $. Therefore,
	$$ v_3(C_{r,n}) = - \frac{1}{2}w_3(C_{r,n}) = - \frac{nk(k+1)}{4}. $$
	
	For the case that $ r < 0 $ or $ n < 0 $, the formula can also be verified due to the facts that  $ C_{r,n} = C_{-r-1, n} $, $ C_{r,n}^* = C_{r,-n} $ and $ v_3(K) = - v_3(K^*) $.
	
	For the knot $ KT_{r, n} $, since $ v_3(K) $ is determined by its Jones polynomial, which is invariant under mutation, $ v_3(KT_{r,n}) = v_3(C_{r,n}) $ for all $ r, n $.
\end{proof}
\begin{remark}
	One can also compute $ w_3(KT_{r,n}) $ directly in the same way. The only difference is the two components of pretzel link $ P(r+1, -r, r, -r-1) $ are a unknot and a connected sum $ T_{2, r+1} \# T_{2, -r-1} $ or $ T_{2,r} \# T_{2,-r} $, depending on the parity of $ r $. And then, $ a_2(K \# K') = a_2(K) + a_2(K') $ follows the fact that $ \nabla_{K \# K'}(z) = \nabla_{K}(z) \cdot \nabla_{K'}(z) $.
\end{remark}

\begin{example}
	The knot $ K11n34 $ in Hoste-Thistlethwaite table \cite{Atlas} is the mirror of the original Conway knot $ C_{2,1} $, that is, $ K11n34 = C_{2,-1} $. The Jones polynomial is
	$$ V(q) = -q^4+2 q^3-2 q^2+2 q+ q^{-2} -2 q^{-3} +2 q^{-4} -2 q^{-5} + q^{-6}. $$
	Direct calculation of the derivatives of $ V $ at $ q = 1 $ gives that $ V''(1) = 0 $ and $ V'''(1) = -72 $. So, we have
	$$ v_3 = -\frac{1}{48}V''(1) - \frac{1}{144}V'''(1) = \frac{1}{2}. $$
\end{example}

Note that $ \lfloor\frac{r}{2}\rfloor = 0 $ iff $ r = 0 $ or $ 1 $, and $ \lfloor\frac{r}{2}\rfloor = -1 $ iff $ r = -1 $ or $ -2 $. We have the following corollary immediately.
\begin{cor}\label{trivial}
	Let $  K $ belong to one of the families $ KT_{r,n} $ and $ C_{r,n} $. $ v_3(K) = 0 $ if and only if $ K $ is isotopic to the unknot, i.e., $ r \in \{0, \pm 1, -2 \} $ or $ n = 0 $.
\end{cor}
\

Now we can verify the purely and chirally cosmetic surgeries on these knots.
\begin{cor}{\cite[Theorem 2.]{BGLX20}}
	The purely cosmetic surgery conjecture is true for all nontrivial Kinoshita-Terasaka and Conway knots.
\end{cor}
\begin{proof}
	This is a direct consequence of Corollary \ref{trivial} and the following theorem.
\end{proof}
\begin{thm}{\cite[Theorem 3.5.]{IW19}}
	If a knot $ K $ has the finite type invariant $ v_2(K) \neq 0 $ or $ v_3(K) \neq 0 $, then $ S^3_r(K) \not \cong S^3_s(K) $ when $ r \neq s $.
\end{thm}

\begin{proof}[Proof of Theorem \ref{main_thm}]
	Suppose $ K $ be a nontrivial member of $ KT_{r,n} $ or $ C_{r,n} $. From Corollary \ref{trivial}, we have $ v_3(K) \neq 0 $. By definition,
	$$ O(K) = \left|\frac{7a_2(K)^2 - a_2(K) - 10a_4(K)}{4v_3(K)}\right| = 0. $$
	As a consequence of Theorem \ref{key} (\cite[Theorem 1.10.]{IIS21}), $ K $ has no chirally cosmetic surgeries.
\end{proof}

\end{document}